\pdfoutput=1
\RequirePackage{ifpdf}
\ifpdf 
\documentclass[pdftex]{sigma}
\else
\documentclass{sigma}
\fi

\DeclareMathOperator{\res}{res}
\DeclareMathOperator{\id}{id}
\DeclareMathOperator{\sgn}{sgn}

\numberwithin{equation}{section}

\newtheorem{Theorem}{Theorem}[section]
\newtheorem{Lemma}[Theorem]{Lemma}
\newtheorem{Proposition}[Theorem]{Proposition}
{ \theoremstyle{definition}
\newtheorem{Remark}[Theorem]{Remark} }

\begin{document}

\allowdisplaybreaks

\newcommand{\arXivNumber}{1612.01856}

\renewcommand{\PaperNumber}{026}

\FirstPageHeading

\ShortArticleName{Another Approach to Juhl's Conformally Covariant Dif\/ferential Operators}

\ArticleName{Another Approach to Juhl's Conformally Covariant\\ Dif\/ferential Operators from $\boldsymbol{S^n}$ to $\boldsymbol{S^{n-1}}$}

\Author{Jean-Louis CLERC}

\AuthorNameForHeading{J.-L.~Clerc}

\Address{Institut Elie Cartan de Lorraine, Universit\'e de Lorraine, France}
\Email{\href{mailto:jean-louis.clerc@univ-lorraine.fr}{jean-louis.clerc@univ-lorraine.fr}}

\ArticleDates{Received December 07, 2016, in f\/inal form April 11, 2017; Published online April 19, 2017}

\Abstract{A family $({\mathbf D}_\lambda)_{\lambda\in \mathbb C}$ of dif\/ferential operators on the sphere~$S^n$ is constructed. The operators are conformally covariant for the action of the subgroup of conformal transformations of $S^n$ which preserve the smaller sphere $S^{n-1}\subset S^n$. The family of conformally covariant dif\/ferential operators from~$S^n$ to~$S^{n-1}$ introduced by A.~Juhl is obtained by composing these operators on~$S^n$ and taking restrictions to~$S^{n-1}$. }


\Classification{58J70; 43A85}

\section{Introduction}

Let $S=S^{n}$ be the $n$-dimensional sphere in $\mathbb R^{n+1}$ and let $G= {\rm SO}_0(1,n+1)$ be (the neutral component of) the group of conformal transformations of~$S$. Let $S'\simeq S^{n-1}$ be the subspace of points of~$S$ with vanishing last coordinate ($x_n=0$ in our notation) and let $G'\simeq {\rm SO}_0(1,n)$ be the conformal group of $S'$, viewed as the subgroup of~$G$ which stabilizes~$S'$. Let $(\pi_\lambda)_{\lambda \in \mathbb C}$ be the scalar principal series of representations of~$G$ acting on~$C^\infty(S)$. Denote by $\pi_{\lambda \vert G'}$ its restriction to~$G'$. Let $(\pi'_\mu)_{\mu\in \mathbb C}$ be the scalar principal series of~$G'$ acting on~$C^\infty(S')$.

In \cite{j} A.~Juhl has constructed a family ${\mathcal D}_N(\lambda)_{\lambda\in \mathbb C, N\in \mathbb N}$ of dif\/ferential operators from $C^\infty(S)$ into $C^\infty(S')$, which are intertwining operators between ${\pi_\lambda}_{\vert G'}$ and $\pi'_{\lambda+N}$.\footnote{Our $\lambda$ corresponds to $-\lambda$ in Juhl's notation.} Later, these operators were obtained by T.~Kobayashi and B.~Speh in~\cite{ks} as residues of a meromorphic family of \emph{symmetry breaking operators} associated to the restriction problem for the pair $(G,G')$. A third point of view was proposed by T.~Kobayashi and M.~Pevzner in~\cite{kp-I,kp-II}, based on the $F$-method. Similar operators were recently constructed for dif\/ferential forms on spheres~\cite{fjs, kkp}.

The new approach to Juhl's operators which I present in this article follows a method that I~used for similar problems, in the context of the restriction problem for a pair $(G\times G,G')$ where $G'=G$ embedded diagonally in~$G\times G$. I was inf\/luenced by a reminiscence of the \emph{$\Omega$-process} which yields both the \emph{transvectants} and the \emph{Rankin--Cohen brackets}. These operators may be viewed as covariant bi-dif\/ferential operators for the group ${\rm SL}(2,\mathbb R)$, or symmetry breaking dif\/ferential operators from ${\rm SL}(2,\mathbb R) \times {\rm SL}(2,\mathbb R)$ to its diagonal subgroup. For a~presentation of these classical results see Section~5 of~\cite{c} for a quick overview or~\cite{o} for a~thorough exposition of the transvectants.

 The new method was introduced in a collaboration with R.~Beckmann for the conformal group of the sphere (see \cite{bc}) and the scalar principal series, then for $G={\rm SL}(2n,\mathbb R)$ and the degenerate principal series acting on the Grasmmannian ${\rm Gr}(n,2n;\mathbb R)$ (see~\cite{c}).

 The f\/irst step of the method, for the present case, is to introduce the multiplication by $x_n$, viewed as an operator $M$ on $C^\infty(S)$. The operator $M$ is a ``universal'' $G'$-intertwining operator, in the sense that, for any $\lambda\in \mathbb C$, the operator $M$ intertwines ${\pi_\lambda}_{\vert G'}$ and ${\pi_{\lambda-1}}_{\vert G'}$. Next recall the family of Knapp--Stein operators $(I_\lambda)_{\lambda\in \mathbb C}$ which are $G$-intertwining operators with respect to $(\pi_\lambda,\pi_{n-\lambda})$. The operator\footnote{For technical reasons, a normalizing factor is introduced, see~\eqref{widetildeD}.}
 \begin{gather*}{\mathbf D}_\lambda = I_{n-\lambda-1}\circ M\circ I_\lambda\end{gather*}
obtained by twisting $M$ by the appropriate Knapp--Stein intertwining operators
is clearly an intertwining operator with respect to $(\pi_{\lambda \vert G'},\pi_{\lambda+1\vert G'})$. Our main result (see Theorem~\ref{main-theorem}) is that $ {\mathbf D}_\lambda$ is a \emph{differential } operator. The proof is obtained in the non compact realization of the principal series (passing from $S^n$ to $\mathbb R^n$ by a conformal map) and uses Euclidean Fourier transform.

The construction of conformally covariant dif\/ferential operators from $S^n$ to $S^{n-1}$ is now easy. For $N$ a non-negative integer, consider
\begin{gather*}{\mathbf D}_{N,\lambda} = {\mathbf D}_{\lambda+N-1}\circ \cdots \circ {\mathbf D}_{\lambda+1}\circ {\mathbf D}_{\lambda}\qquad
\text{or}\qquad {\mathbb D}_{N,\lambda} = I_{n-\lambda-N}\circ M^N\circ I_\lambda.
\end{gather*}
The two families of dif\/ferential operators on $S$ (which coincide up to a meromorphic function of~$\lambda$) are covariant with respect to $(\pi_{\lambda \vert G'},\pi_{\lambda+N\vert G'})$. Finally, let
\begin{gather*}{\mathbf D}_N(\lambda) = \res\circ {\mathbf D}_{N,\lambda},
\end{gather*}
where $\res$ is the restriction map from $C^\infty(S)$ to $C^\infty(S')$. The operator ${\mathbf D}_N(\lambda)$ is a dif\/ferential operator from $S$ to~$S'$ which is covariant with respect to $(\pi_{\lambda \vert G'} ,\pi'_{\lambda+N})$. The family ${\mathbf D}_N(\lambda)_{\lambda\in \mathbb C, N\in \mathbb N}$ essentially coincides with Juhl's family.

The operator $\mathbf D_\lambda$ has a simple expression in the non compact picture, see~\eqref{Dlambda}. It is tempting to f\/ind a more direct approach to this operator. This is achieved in the last section, by using yet another realization of the principal series, sometimes called the \emph{ambient space} realization. The way the operator is constructed is much simpler, and it is then easy to determine its expression in the non compact picture (recovering the expression of $\mathbf D_\lambda$ on $\mathbb R^n$, see Proposition~\ref{Dlambdanc}), but also in the compact realization (see Proposition~\ref{Dlambdac}), that is to say as a $G'$-conformally covariant dif\/ferential operator on~$S$. Some generalization of these formul\ae\ in the realm of conformal geometry on a Riemannian manifold seems plausible.

\section[The principal series of ${\rm SO}_0(1,n+1)$ and the Knapp--Stein intertwining operators]{The principal series of $\boldsymbol{{\rm SO}_0(1,n+1)}$ \\ and the Knapp--Stein intertwining operators}

Let $E$ be a Euclidean space of dimension $n+1$, and choose an orthonormal basis $\{ e_0, e_1,\dots, e_n\}$. Let $S=S^n$ be the unit sphere of~$E$, i.e.,
\begin{gather*}
S=\big\{ x=(x_0,x_1,\dots, x_n), \ x_0^2+x_1^2+\dots +x_n^2=1\big\}.
\end{gather*}
Let $\mathbf E$ be the vector space $\mathbb R\oplus E$, with the Lorentzian quadratic form
\begin{gather*}Q(\mathbf x) = [(t,x), (t,x)] = t^2-\vert x\vert^2\qquad \text{for} \quad \mathbf x=(t,x), \quad t\in \mathbb R, \quad x\in E .
\end{gather*}
For ${\bf x}=(t,x)\in\bold E$, we let
\begin{gather*}t({\bf x}) = t,\qquad {\bf x}_E=x .
\end{gather*}

The space of isotropic lines $\mathcal S$ in $\mathbf E$ can be identif\/ied with $S$ by the map
\begin{gather*} S\ni x \longmapsto d_x = \mathbb R(1,x)\in \mathcal S, \qquad \mathcal S \ni d \longmapsto d\cap \{ t=1\}.
\end{gather*}

Let $G= {\rm SO}_0(1,n+1)$ be the connected component of the group of isometries of $\mathbf E$. Then~$G$ acts on $\mathcal S$ and this action can be transferred to an action on~$S$. More explicitly, if $x=(x_0, x_1,\dots, x_n)\in S$, and $g\in G$, observe that $t(g(1,x)) >0$ and def\/ine $g(x)\in S$ by
\begin{gather*}
\big(1,g(x)\big) = t\big(g.(1,x)\big)^{-1} g.(1,x).
\end{gather*}
Set, for $g\in G$ and $x\in S$
 \begin{gather*}
 \kappa(g,x) = t\big(g.(1,x)\big)^{-1}.
 \end{gather*}
Clearly $\kappa(g,x)$ is a smooth, strictly positive function on $G\times S$. Moreover $\kappa(g,x)$ satisf\/ies the \emph{cocycle property}: for any $g_1$, $g_2$ and any $x\in S$,
\begin{gather*}
\kappa(g_1g_2,x) = \kappa\big(g_1, g_2(x)\big) \kappa(g_2,x).
\end{gather*}
This action of $G$ on $S$ is known to be \emph{conformal}. For $g\in G$, $x\in S$ and $\xi$ an arbitrary tangent vector to~$S$ at~$x$
\begin{gather*}
\vert Dg(x) \xi \vert = \kappa(g,x) \vert \xi \vert,
\end{gather*}
and hence $\kappa(g,x)$ is the \emph{conformal factor} of~$g$ at~$x$.

Associated to the action of $G$ on $S$ there is a family of representations on $C^\infty(S)$, which, from the point of view of harmonic analysis is the \emph{scalar principal series} of $G$. For $\lambda\in \mathbb C$, $g\in G$ and $f\in C^\infty(S)$, let
\begin{gather*}
\pi_\lambda(g) f(x) = \kappa\big(g^{-1},x\big)^\lambda f\big(g^{-1}(x)\big).
\end{gather*}
The formula def\/ines a (smooth) representation $\pi_\lambda$ of $G$ on $C^\infty(S)$.

The Knapp--Stein intertwining operators are a major tool in harmonic analysis of $G$ (as of any semi-simple Lie group, see, e.g.,~\cite{kn}). For $\lambda\in \mathbb C$ and $f\in C^\infty(S)$, let
\begin{gather}\label{KS}
I_\lambda f(x) =\frac{1}{\Gamma(\lambda -\frac{n}{2})} \int_S \vert x-y\vert^{-2n+2\lambda} f(y) dy,
\end{gather}
where $dy$ stands for the Lebesgue measure on $S$ induced by the Euclidean structure. For \smash{$\operatorname{Re}\lambda>\frac{n}{2}$}, this formula def\/ines a continuous operator~$ I_\lambda$ on~$C^\infty(S)$.
\begin{Proposition} \quad
\begin{enumerate}\itemsep=0pt
\item[$i)$] The definition \eqref{KS} can be analytically continued in $\lambda$ to all of $\mathbb C$.

\item[$ii)$] The analytic continuation yields a holomorphic family of operators $I_\lambda$ on $C^\infty(S)$, which satisfy the intertwining relation
\begin{gather}\label{intwKS}
\forall\, g\in G,\qquad I_\lambda \circ \pi_\lambda(g) = \pi_{n-\lambda}(g)\circ I_\lambda.
\end{gather}
\end{enumerate}
\end{Proposition}
The following complementary result will be needed later.
\begin{Proposition}
For any $\lambda\in \mathbb C$
\begin{gather}\label{inverseJ}
I_\lambda \circ I_{n-\lambda} = \frac{\pi^n}{\Gamma(\lambda)\Gamma(n-\lambda)} \id.
\end{gather}
\end{Proposition}

The next result corresponds to reducibility points for the scalar principal series. Let $\mathcal P(S)$ be the space of restrictions to $S$ of polynomial functions on $E$, and for $k\in \mathbb N$, let $\mathcal P_k$ be the space of restrictions to $S$ of polynomials on $E$ of degree $\leq k$. Finally, let $\mathcal P_k^\perp$ be the subspace of $C^\infty(S)$ given by
\begin{gather*}\mathcal P_k^\perp = \left\{f\in C^\infty(S), \, \int_Sf(x) p(x) dx=0, \ \text{for any} \ p\in \mathcal P_k \right\}.
\end{gather*}.

\begin{Proposition} \quad
\begin{enumerate}\itemsep=0pt
\item[$i)$] Let $ \lambda = n+k, k\in \mathbb N$. Then
\begin{gather}\label{kerim1}
\operatorname{Im}(I_{n+k}) = \mathcal P_k,\qquad \operatorname{Ker}(I_{n+k}) = \mathcal P_k^\perp.
\end{gather}
\item[$ii)$] Let $\lambda = -k, k\in \mathbb N$. Then
\begin{gather}\label{kerim2}
 \operatorname{Ker}(I_{-k} )= \mathcal P_k,\qquad \operatorname{Im}(I_{-k}) = \mathcal P_k^\perp.
\end{gather}
\end{enumerate}
\end{Proposition}

\section[Construction of the family $\widetilde{\mathbf D}_\lambda$, $\lambda \in \mathbb C$]{Construction of the family $\boldsymbol{\widetilde{\mathbf D}_\lambda}$, $\boldsymbol{\lambda \in \mathbb C}$}

Now let $E' = \{ x\in E,\, x_n=0\}$ and $S'=S\cap E'$. Then $S'$ is an $(n-1)$-dimensional sphere. Let~$G'$ be the subgroup of elements of $G$ of the form
\begin{gather*}g = \begin{pmatrix} &&&0\\&g'&&\vdots\\&&&0\\0&\dots&0&1\end{pmatrix},\qquad g'\in {\rm SO}_0(1,n).
\end{gather*}
Clearly, $G'$ is a subgroup of $G$, isomorphic to ${\rm SO}_0(1,n)$. Elements of $G'$ preserve the hyperplane $\{x_n=0\}$ in $\mathbf E$ and hence the action of $G'$ on $S$ preserves $S'$.

For $x\in E$, write $x=(x',x_n)$, with $x'\in \mathbb R^n$. For $g\in G'$, \begin{gather*}
g (1,x) = g(1,x',x_n)=\big(g'.(1,x'),x_n\big).
\end{gather*}
If $x\in S$, the last equation can be rewritten as
\begin{gather*}
\kappa(g,x)^{-1} \big(1,g(x)\big) =\big(g'.(1,x'),x_n\big),
\end{gather*}
so that
\begin{gather}\label{covxn}
g(x)_n =\kappa(g,x) x_n.
\end{gather}
In the sequel, the distinction between $g$ and $g'$ in the notation is abandoned, the context providing the correct interpretation.

 Let $M$ be the operator def\/ined on $C^\infty(S)$ by
\begin{gather*}Mf(x) = x_n f(x), \qquad f\in C^\infty(S).
\end{gather*}
\begin{Proposition}\label{intwM}
The operator $M$ satisfies
\begin{gather}\label{Mintw}
 \forall\, g\in G'\qquad M\circ \pi_\lambda(g) = \pi_{\lambda-1}(g) \circ M.
\end{gather}
\end{Proposition}

\begin{proof} This is an immediate consequence of \eqref{covxn}.
\end{proof}

Next let ${\mathbf D}_\lambda$ be the operator on $C^\infty(S)$ def\/ined by
\begin{gather*}{\mathbf D}_\lambda = I_{n-\lambda-1} \circ M\circ I_\lambda,
\end{gather*}
which corresponds to the following diagram
\begin{gather*}
\begin{CD}
C^\infty(S) @> {\mathbf D}_\lambda>> C^\infty(S) \\
@VV I_\lambda V @AA I_{n-\lambda-1} A\\
C^\infty(S) @>M> >C^\infty(S).
\end{CD}
\end{gather*}
As a consequence of the intertwining property of the Knapp--Stein operators~\eqref{intwKS} and Proposition~\ref{intwM}, ${\mathbf D}_\lambda$ satisf\/ies for $g\in G'$
\begin{gather}\label{intwD}
{\mathbf D}_\lambda \circ \pi_\lambda(g) = \pi_{\lambda+1}(g) \circ{\mathbf D}_\lambda.
\end{gather}
Otherwise said, the operator ${\mathbf D}_\lambda$ is covariant with respect to $({\pi_\lambda}_{\vert G'},{ \pi_{\lambda+1}}_{\vert G'})$.
\begin{Theorem}\label{main-theorem}
The operator ${\mathbf D}_\lambda$ is a differential operator on $S$.
\end{Theorem}
The proof of Theorem~\ref{main-theorem} will be given at the end the next section.

\begin{Proposition} Let $\lambda\in (n+\mathbb N)\cup (-1-\mathbb N)$. Then ${\mathbf D}_\lambda =0$.
\end{Proposition}

\begin{proof} Let f\/irst $\lambda=n+k$ for some $k\in \mathbb N$. Then $I_\lambda = I_{n+k}$, and by~\eqref{kerim1} $\operatorname{Im}(I_\lambda)=\mathcal P_k$. Next $\operatorname{Im}(M\circ I_\lambda)\subset \mathcal P_{k+1}$. Now $I_{n-\lambda-1} = I_{-k-1}$ and using \eqref{kerim2}, $I_{n-\lambda-1} \circ M\circ I_{\lambda} = 0$.

Now let $\lambda=-k$, with $k\geq 1$. Then $I_\lambda= I_{-k}$ and by~\eqref{kerim2}, $\operatorname{Im}(I_\lambda)=\mathcal P_k^\perp$. Next $\operatorname{Im}(M\circ I_\lambda)\subset \mathcal P_1\mathcal P_k^\perp\subset\mathcal P_{k-1}^\perp$. Now $I_{n-\lambda-1} = I_{n+k-1}$ which using~\eqref{kerim1} implies $I_{n-\lambda-1} \circ M\circ I_{\lambda} = 0$.
\end{proof}

To compensate for these zeroes of ${\mathbf D}_\lambda$, introduce
\begin{gather}\label{widetildeD}
\widetilde {\mathbf D}_\lambda = \Gamma(\lambda+1) \Gamma(n-\lambda) {\mathbf D}_\lambda
\end{gather}
for $\lambda\notin (n+\mathbb N)\cup (-1-\mathbb N)$ and extend continuously to all of $\mathbb C$ to get a holomorphic family ${(\widetilde {\mathbf D}_\lambda)}_{\lambda\in \mathbb C}$ of dif\/ferential operators on~$S$ covariant with respect to $(\pi_{\lambda\vert G'}, \pi_{\lambda+1\vert G'})$.

\section[The expression of ${\widetilde {\mathbf D}}_\lambda$ in the non-compact picture]{The expression of $\boldsymbol{{\widetilde {\mathbf D}}_\lambda}$ in the non-compact picture}

Consider the point $-{\mathbf 1}=(-1,0,\dots,0)\in S$. The stereographic projection with source at $-{\mathbf 1}$ provides a dif\/feomorphism from $S\setminus\{-\mathbf 1\} $ onto the hyperplane~$\{ x_n=1\}$. The inverse map (up to a scaling by a factor 2) $c\colon \mathbb R^n\longrightarrow S$ is given by
\begin{gather}\label{explicitc}
c(\xi) = \begin{pmatrix} \dfrac{1-\vert \xi\vert^2}{1+\vert \xi\vert ^2}\vspace{1mm}\\ \dfrac{2\xi_1}{1+\vert \xi\vert ^2}\\ \vdots \\ \dfrac{2\xi_n}{1+\vert \xi\vert ^2}\end{pmatrix}.
\end{gather}
When using this local chart on $S$, we refer to the \emph{non-compact picture}, as a reference to semi-simple harmonic analysis.

Geometric considerations (or an elementary computation) show that, for $\xi,\eta\in \mathbb R^n$
\begin{gather*}\vert c(\xi)-c(\eta)\vert^2 = \kappa(c,\xi) \vert \xi-\eta\vert^2 \kappa(c,\eta),
\end{gather*}
where, for $\xi\in \mathbb R^n$, we set
\begin{gather*}
\kappa(c,\xi) = 2\big(1+\vert \xi\vert^2\big)^{-1}.
\end{gather*}
There is an inf\/initesimal version of this result, namely
\begin{gather*}
\vert Dc(\xi)\eta \vert= \kappa(c,\xi)\vert \eta\vert
\end{gather*}
for $\xi,\eta\in \mathbb R^n$. This last statement shows that $c$ is conformal from $\mathbb R^n$ with its standard Euclidean structure into~$S$.

The action of $g$ on $S$ can be transferred as a (rational) action of~$G$ on $\mathbb R^n$, namely $c^{-1}\circ g\circ c$. For notational convenience, we still denote this action on~$\mathbb R^n$ by $(g,\xi) \longmapsto g(\xi)$, $g\in G$, $\xi\in \mathbb R^n$. As the map~$c$ is conformal, the transferred action of~$G$ on~$\mathbb R^n$ is still conformal. For $g\in G$ def\/ined at $\xi\in \mathbb R^n$, we let $\kappa(g,\xi)$ be the corresponding conformal factor of~$g$ at~$\xi$.

Let $\lambda\in \mathbb C$. For $f\in C^\infty(S)$ let $C_\lambda(f)$ be def\/ined by
\begin{gather*} C_\lambda(f)(\xi) = \kappa(c,\xi)^\lambda f(c(\xi)),\qquad \xi \in \mathbb R^n
\end{gather*}
and let $\mathcal H_\lambda$ be the image of $C_\lambda$. It is easily proved that
\begin{gather*}\mathcal S\big(\mathbb R^n\big)\subset \mathcal H_\lambda\subset \mathcal S'\big(\mathbb R^n\big),
\end{gather*}
where $\mathcal S(\mathbb R^n)$ stands for the Schwartz space on $\mathbb R^n$ and $\mathcal S'(\mathbb R^n)$ for its dual, the space of tempered distributions.

The representation $\pi_\lambda$ can be transferred in the non-compact model, using $C_\lambda$ as intertwining map, i.e., set \begin{gather*}\rho_\lambda(g)=C_\lambda\circ \pi_\lambda(g)\circ C_\lambda^{-1}.\end{gather*} Using the cocycle property of $\kappa$, $\rho_\lambda$ can be realized as
\begin{gather*}\rho_\lambda(g)f(\xi) = \kappa\big(g^{-1},\xi\big)^\lambda f\big(g^{-1}(\xi)\big),
\end{gather*}
where $f\in \mathcal H_\lambda$ and $g\in G$.

Similarly, the Knapp--Stein operators can be transferred to the non-compact picture. For $s\in \mathbb C$, consider the expression
\begin{gather*}h_s(\xi) = \frac{1}{\Gamma\big(\frac{n}{2} +\frac{s}{2}\big)} \vert \xi\vert^s,\qquad \xi\in \mathbb R^n.
\end{gather*}
For $\operatorname{Re}(s)>-n$, $h_s$ is locally summable with moderate growth at inf\/inity, hence def\/ines a~tempered distribution. The ($\mathcal S'(\mathbb R^n)$-valued) function $s \mapsto h_s$ can be extended by analytic con\-ti\-nuation to~$\mathbb C$ and the $\Gamma$ factor in the def\/inition of~$h_s$ is so chosen that it extends as an \emph{entire} function with values in $\mathcal S'(\mathbb R^n)$ (for more details see, e.g.,~\cite{gs}).

For $\lambda\in \mathbb C$, the Knapp--Stein operator $J_\lambda$ is given by
\begin{gather*}
J_\lambda f = h_{-2n+2\lambda} \star f,
\end{gather*}
or more concretely
\begin{gather*}
J_\lambda f(\xi) = \frac{1}{\Gamma\big(\lambda-\frac{n}{2}\big)} \int_{\mathbb R^n} \vert \xi-\eta\vert^{-2n+2\lambda} f(\eta) d\eta.
\end{gather*}

As for any $s\in \mathbb C$ $h_s$ is a tempered distribution, $J_\lambda$ maps $\mathcal S(\mathbb R^n)$ into $\mathcal S'(\mathbb R^n)$.

\begin{Proposition} Let $\lambda\in \mathbb C$. Then for $f\in \mathcal S(\mathbb R^n)$
\begin{gather*}J_\lambda f = \big(C_{n-\lambda}\circ I_\lambda\circ C_\lambda^{-1}\big) f .
\end{gather*}
\end{Proposition}
\begin{proof} As $\mathcal S(\mathbb R^n)\subset \mathcal H_\lambda\subset \mathcal S'(\mathbb R^n)$, both sides are well-def\/ined and belong to $\mathcal S'(\mathbb R^n)$. For $\operatorname{Re}\lambda>\frac{n}{2}$, both sides are given by convergent integrals, and the equality is proved by a change of variable. The general case follows by analytic continuation. The intertwining property of the Knapp--Stein operators can be formulated in the following way.
\end{proof}

\begin{Proposition}\label{intwJ}
 Let $f\in C_c^\infty(S)$ and let $g\in G$ such that $g^{-1}$ is defined on $\operatorname{Supp}(f)$. Then
\begin{gather*}J_{\lambda}\big(\rho_\lambda(g) f \big) = \rho_{n-\lambda}(g)\big( J_\lambda f \big),
\end{gather*}
where the two sides of the equation are viewed as tempered distributions on~$\mathbb R^n$.
\end{Proposition}
\begin{proof} The condition implies that both $f$ and $\rho_\lambda(g)f$ are contained in $\mathcal S(\mathbb R^n)$. Hence
\begin{gather*}
J_{\lambda}\big(\rho_\lambda(g) f \big) = \big(C_{n-\lambda}\circ I_\lambda\circ C_\lambda^{-1}\big)\big( \rho_\lambda(g)f\big)\\
\hphantom{J_{\lambda}\big(\rho_\lambda(g) f \big)}{}
= (C_{n-\lambda}\circ I_\lambda)\big(\pi_\lambda(g)C_\lambda^{-1}f\big)
=\big(C_{n-\lambda}\circ \pi_{n-\lambda}(g)\big)\circ\big( I_\lambda \circ C_\lambda^{-1}\big)f\\
\hphantom{J_{\lambda}\big(\rho_\lambda(g) f \big)}{}
= \rho_{n-\lambda}(g)\circ \big(C_{n-\lambda}\circ I_\lambda\circ C_{\lambda}^{-1}\big) f= \rho_{n-\lambda}(g)(J_\lambda f).\tag*{\qed}
\end{gather*}\renewcommand{\qed}{}
\end{proof}

The following formul{\ae} will be needed in the sequel
\begin{gather}\label{xi2}
\vert \xi\vert^2 h_s(\xi)= \frac{n+s}{2} h_{s+2}(\xi), \\
\label{dxin}
\frac{\partial}{\partial \xi_n} h_s (\xi)= \frac{2s}{n+s-2} \xi_n h_{s-2}(\xi),
\end{gather}
where at $s=-n+2$, the last formula has to be understood by analytic continuation.

As the pole of the stereographic projection has been chosen in $S'$, the map~$c$ maps the hyperplane $\{\xi_n=0\}$ into~$S'$. It allows to transfer the map $M$ to the non-compact picture.

\begin{Lemma}
 Let $g\in G'$ be defined at $\xi\in \mathbb R^n$. Then
\begin{gather}\label{covxin}
g(\xi)_n = \kappa(g,\xi)\xi_n.
\end{gather}
\end{Lemma}

\begin{proof} Let $\xi\in \mathbb R^n$ and let $x=c(\xi)\in S\setminus\{-\mathbf{1}\}$. Then
\begin{gather*}c(\xi)_n = \kappa(c,\xi)\xi_n,\qquad g(x)_n= \kappa(g,x) x_n,\qquad c^{-1}(x)=\kappa\big(c^{-1},x\big)x_n
\end{gather*}
the f\/irst equality by \eqref{explicitc}, the second by~\eqref{covxn}, and the third also by \eqref{explicitc}. As $\kappa$ satisf\/ies a~cocycle relation, we get
\begin{gather*}\big(\big(c^{-1}\circ g\circ c\big) (\xi)\big)_n = \kappa\big(c^{-1}\circ g \circ c, \xi\big) \xi_n,
\end{gather*}
which gives \eqref{covxin}.
\end{proof}

\begin{Lemma}\label{MClambda}
 Let $\lambda\in \mathbb C$ and $f\in C^\infty(S)$. Then
\begin{gather*}C_{\lambda-1}(Mf)(\xi) = \xi_n C_\lambda(f)(\xi), \qquad \xi\in \mathbb R^n.
\end{gather*}
\end{Lemma}

\begin{proof} Let $\xi=(\xi',\xi_n)$. By \eqref{explicitc}, $c(\xi)_n= \kappa(c,\xi)\xi_n$, so that
\begin{gather*}C_{\lambda-1}(Mf)(\xi) = \kappa(c,\xi)^{\lambda-1} Mf\big(c(\xi)\big)= \kappa(c,\xi)^\lambda \xi_n f\big(c(\xi)\big) = \xi_n C_\lambda(f)(\xi).\tag*{\qed}
\end{gather*}\renewcommand{\qed}{}
\end{proof}

Abusing notation, $M$ will be used for the operator (on $C^\infty(\mathbb R^n)$ say) of multiplication by $\xi_n$. The operator $M$ maps $\mathcal S(\mathbb R^n)$ (resp.~$\mathcal S'(\mathbb R^n))$ into $\mathcal S(\mathbb R^n)$ (resp.~$\mathcal S'(\mathbb R^n))$, and for any $\lambda\in \mathbb C$, the operator $M$ maps $\mathcal H_\lambda$ into $\mathcal H_{\lambda-1}$ (Lemma~\ref{MClambda}).

\begin{Proposition}\label{intwMnc}
 Let $\lambda\in \mathbb C$. The operator $M\colon \mathcal H_\lambda\longrightarrow\mathcal H_{\lambda-1}$ satisfies
\begin{gather*}\forall\, g\in G',\qquad M\circ \rho_\lambda(g)= \rho_{\lambda-1}(g)\circ M.
\end{gather*}
Otherwise said, the operator $M$ intertwines the representations ${\pi_\lambda}_{\vert G'}$ and ${\pi_{\lambda-1}}_{\vert G'}$.
\end{Proposition}
\begin{proof} Let $f\in \mathcal H_\lambda$. Then
\begin{gather*}\big(M\circ \rho_\lambda(g)\big) f (\xi) = \xi_n \kappa\big(g^{-1}, \xi\big)^\lambda f\big(g^{-1}(\xi)\big)= \kappa\big(g^{-1},\xi\big)^{\lambda-1}\big(g^{-1}(\xi)\big)_n
f\big(g^{-1}(\xi)\big)\\
\hphantom{\big(M\circ \rho_\lambda(g)\big) f (\xi)}{}
= \rho_{\lambda-1}(g) (Mf)\big(g^{-1}(\xi)\big)
\end{gather*}
and the statement follows.
\end{proof}

Having introduced the non-compact version of the main ingredients, we observe that the Knapp--Stein operators are convolution operators, whereas~$M$ is the multiplication by an elementary polynomial. So the Fourier transform is well-f\/itted for computations in this context. Def\/ine the Fourier transform on $\mathbb R^n$ as usual by
\begin{gather*}\widehat f(\eta) = \int_{\mathbb R^n} e^{i\langle \eta, \xi\rangle} f(\xi) d\xi
\end{gather*}
initially for functions in $\mathcal S(\mathbb R^n)$ and extend by duality to $\mathcal S'(\mathbb R^n)$.

The Fourier transform of $h_s$ is given by
\begin{gather*}
\widehat h_s = 2^{n+s} \pi^{\frac{n}{2}} h_{-n-s}.
\end{gather*}
For this result see, e.g.,~\cite{gs}.

Thanks to the above observations, it is possible to def\/ine the composition $M\circ J_\lambda$ as an operator from~$\mathcal S(\mathbb R^n)$ into $\mathcal S'(\mathbb R^n)$.

\begin{Lemma} For $f\in \mathcal S(\mathbb R^n)$,
\begin{gather}
\big((M\circ J_\lambda) f\big)^{\widehat{\ }}(\eta) = -i \pi^{\frac{n}{2}} 2^{-n+2\lambda} \left(h_{n-2\lambda} (\eta)\frac{\partial \widehat f}{\partial \eta_n}(\eta)+ \frac{n-2\lambda}{n-\lambda-1} \eta_n h_{n-2-2\lambda}(\eta)\widehat f(\eta)\right).\!\!\!\!\label{FourierMJ}
\end{gather}
\end{Lemma}

\begin{proof}
As observed earlier, the Knapp--Stein operator $J_\lambda$ is a convolution operator on $\mathbb R^n$, so that
\begin{gather*}(J_\lambda f)^{\widehat \ }(\eta) = \widehat h_{-2n+2\lambda}(\eta) \widehat f(\eta)
=2^{-n+2\lambda} \pi^{\frac{n}{2}} h_{n-2\lambda}(\eta) \widehat f(\eta).
\end{gather*}
Next, for any distribution $\varphi\in \mathcal S'(\mathbb R^n)$
\begin{gather*}\widehat{M\varphi} = -i\frac{\partial}{\partial \eta_n} \widehat{\varphi}
\end{gather*}
and \eqref{FourierMJ} follows, using \eqref{dxin}.
\end{proof}

The composition $J_{n-\lambda-1}\circ M\circ J_\lambda$ is not well-def\/ined on $\mathcal S(\mathbb R^n)$. However, a formal computation (using again Fourier transforms) can be made and leads to a dif\/ferential operator, which is at the origin of the def\/inition~\eqref{Dlambda} below. In order to give a rigorous argument, it is necessary to follow an indirect route.

For $\lambda\in \mathbb C$, let $E_\lambda$ be the dif\/ferential operator on $\mathbb R^n$ def\/ined by
\begin{gather}\label{Dlambda}
E_\lambda = (2\lambda-n+2) \frac{\partial }{\partial \xi_n} +\xi_n \Delta,
\end{gather}
where $\Delta = \sum\limits_{j=1}^n \frac{\partial^2}{ \partial \xi_j^2}$ is the usual Laplacian on $\mathbb R^n$. Notice that the operator $E_\lambda$ maps $\mathcal S(\mathbb R^n)$ (resp.~$\mathcal S'(\mathbb R^n)$) into $\mathcal S(\mathbb R^n)$ (resp.~$\mathcal S'(\mathbb R^n)$), so that we may consider the composition $J_{\lambda+1}\circ E_\lambda$.

\begin{Lemma} For $f\in \mathcal S(\mathbb R^n)$,
\begin{gather}
\big((J_{\lambda +1} \circ E_\lambda)f\big)^{\widehat{\ }}(\eta)\nonumber \\
{} =-i2^{-n+2+2\lambda}\pi^{\frac{n}{2}}\left((\lambda-n+1)h_{n-2\lambda}(\eta) \frac{ \partial \widehat f }{\partial \eta_n} (\eta) +(2\lambda-n) \eta_n h_{-2\lambda+n-2} (\eta) \widehat f (\eta)\right).\label{FourierJE}
\end{gather}
\end{Lemma}

\begin{proof} Using \eqref{xi2} and \eqref{dxin},
\begin{gather*}\big( E_\lambda f\big)^{\widehat{\ }}(\eta)= (-i)(2\lambda-n+2) \eta_n \widehat f (\eta)+(-i) \frac{\partial}{\partial \eta_n}\big({-}\vert \eta\vert^2 \widehat f (\eta)\big)\\
\hphantom{\big( E_\lambda f\big)^{\widehat{\ }}(\eta)}{} =(-i) \left( (2\lambda-n)\eta_n \widehat f (\eta) -\vert \eta\vert^2 \frac{\partial \widehat f}{\partial \eta_n}(\eta)\right).
\end{gather*}
Next
\begin{gather*}
\big((J_{\lambda+1} \circ E_\lambda)f\big)^{\widehat{\ }}(\eta) = \widehat h_{-2n+2\lambda+2}(\eta) (E_\lambda f)^{\widehat{\ }}(\eta)\\
{} = 2^{-n+2+2\lambda} \pi^{\frac{n}{2}}(-i)\left((2\lambda-n) \eta_n h_{n-2-2\lambda}(\eta) \widehat f (\eta)-(-\lambda+n-1)h_{n-2\lambda}(\eta) \frac{\partial\widehat f}{\partial \eta_n}(\eta)\right).\tag*{\qed}
\end{gather*}
\renewcommand{\qed}{}
\end{proof}

Comparison of \eqref{FourierMJ} and \eqref{FourierJE} yields the next result.
\begin{Proposition}
\begin{gather}\label{MI}
M\circ J_\lambda = \frac{1}{4(\lambda-n+1)} J_{\lambda+1}\circ E_\lambda .
\end{gather}
\end{Proposition}

\begin{Remark}This equality has to be understood as an equality of operators from $\mathcal S(\mathbb R^n)$ into $\mathcal S'(\mathbb R^n)$. For $\lambda= n-1$, $J_{\lambda+1}=J_n$ is equal (up to a constant $\neq 0$) to the operator $f\longmapsto \left(\int_{\mathbb R^n} f(\xi) d\xi\right) 1$. Now for $f\in \mathcal S(\mathbb R^n)$, $\int_{\mathbb R^n} E_\lambda f(\xi) d\xi=0$ as is easily seen by integration by parts. Hence, $J_{\lambda+1} \circ E_\lambda$ vanishes for $\lambda = n-1$, so that~\eqref{MI} has to be interpreted as a residue formula.
\end{Remark}

\begin{Proposition}\label{covE1}
Let $f\in C^\infty_c(\mathbb R^n)$ and assume that $g\in G'$ is such that $g^{-1}$ is defined on $\operatorname{Supp}(f)$. Then
\begin{gather*}
\big(E_\lambda\circ\rho_\lambda(g)\big)f = \big(\rho_{\lambda+1}(g)\circ E_\lambda\big) f.
\end{gather*}
\end{Proposition}

\begin{proof} As a consequence of the intertwining property of $J_\lambda$ (Proposition~\ref{intwJ}) and of~$M$ (Proposition~\ref{intwM}),
\begin{gather*}(M\circ J_\lambda) \big(\rho_\lambda(g)f\big) = \rho_{n-\lambda-1}(g) (M\circ J_\lambda)f.
\end{gather*}
Hence, by \eqref{MI} (assuming for a while that $\lambda\neq n-1$)
\begin{gather*}\big(J_{\lambda+1}\circ E_\lambda\big) \rho_\lambda(g)f = \rho_{n-\lambda-1}(g) \big((J_{\lambda+1}\circ E_\lambda)f\big).
\end{gather*}
Now $\operatorname{Supp}(E_\lambda f)\subset \operatorname{Supp}(f)$, so that $g^{-1}$ is def\/ined on $\operatorname{Supp}(E_\lambda f)$. Hence, by Proposition \ref{intwJ}
\begin{gather*}\big(\rho_{n-\lambda-1}(g) \circ J_{\lambda+1}\big) E_\lambda f=\big(J_{\lambda+1} \circ \rho_{\lambda+1}(g)\big) E_\lambda f,
\end{gather*}
so that
\begin{gather*}J_{\lambda+1}\big( (E_\lambda\circ\rho_\lambda(g))f\big)=J_{\lambda+1} \big( (\rho_{\lambda+1}(g)\circ E_\lambda) f\big).
\end{gather*}
Now, for $\lambda$ generic, the operator $J_{\lambda+1}$ is injective on $\mathcal S(\mathbb R^n)$, hence
\begin{gather*}E_\lambda\circ\rho_\lambda(g)f = \rho_{\lambda+1}(g)\circ E_\lambda f.
\end{gather*}
The general result follows by continuity, as the family $E_\lambda$ depends holomorphically on $\lambda$.
\end{proof}

\begin{proof}[Proof of Theorem~\ref{main-theorem}]
The covariance property of the dif\/ferential operator $E_\lambda$ allows to construct a \emph{global} dif\/ferential operator on $S$ which is expressed in the non-compact picture to~$E_\lambda$. In fact to fully cover the sphere $S$, we only need another chart, which can be chosen as the analog of the map $c$ but constructed from the stereographic projection corresponding to the pole $\mathbf 1=(1,0,\dots, 0)$ instead of~$-{\mathbf 1}$. Consider the element $s$ of $G$ given by
\begin{gather*}s=\begin{pmatrix}1&0&0&0&\dots&0\\0&-1&0&0&\dots&0\\0&0&-1&0&\dots &0\\ 0&0&0&1&&0\\\vdots&\vdots&\vdots &&\ddots&\vdots\\0&0&0&0&\dots&1
\end{pmatrix}.
\end{gather*}
Then $s$ acts on $S$ by
\begin{gather*}s(x) = s\left(\begin{matrix}x_0\\x_1\\x_2\\\vdots\\x_n \end{matrix}\right) = \left(\begin{matrix}-x_0\\-x_1\\x_2\\\vdots\\x_n \end{matrix}\right).
\end{gather*}
In particular, $s$ maps $-\mathbf 1$ to $\mathbf 1$ and preserves $S'$. In the non-compact picture, the map $s$ is def\/ined for $\xi\neq 0$ and is given by
\begin{gather}\label{change}
(\xi_1,\xi_2,\dots, \xi_n) \longmapsto \left(-\frac{\xi_1}{\vert \xi\vert^2},\frac{\xi_2}{\vert \xi\vert^2},\dots,\frac{\xi_n}{\vert \xi\vert^2}\right).
\end{gather}
The two charts $\xi\longmapsto c(\xi)$ and $\xi\longmapsto s\big(c(\xi)\big)$ cover $S$. Their common domain corresponds to $\xi\neq 0$, the change of chart being given by~\eqref{change}, which is the local expression in the non-compact picture of the transform~$s$. So Proposition~\ref{covE1}, when applied to $g=s$ is exactly what is needed to prove that there is a global dif\/ferential operator $\mathbf E_\lambda$ on $S$ which is expressed by $E_\lambda$ in the non-compact model. Clearly $\mathbf E_\lambda$ satisf\/ies
\begin{gather*}\forall\, g\in G',\qquad \mathbf E_\lambda \circ \pi_\lambda(g) = \pi_{\lambda+1}(g) \circ \mathbf E_\lambda.
\end{gather*}
By \eqref{MI},
\begin{gather*}M\circ I_\lambda = \frac{1}{4(\lambda-n+1)}I_{\lambda+1}(g)\circ \mathbf E_\lambda.
\end{gather*}
Compose both sides with $I_{n-\lambda-1}$ and use \eqref{inverseJ} to get
\begin{gather*}
{\mathbf D}_\lambda = \frac{\pi^{-n}}{4(\lambda-n+1)\Gamma(n-\lambda-1)\Gamma(\lambda+1)} \mathbf E_\lambda
\end{gather*} or equivalently
 \begin{gather*}
 \widetilde {\mathbf D}_\lambda = -\frac{1}{4\pi^n} \mathbf E_\lambda.
 \end{gather*}
 This relation implies in particular that $\widetilde {\mathbf D}_\lambda$ is a dif\/ferential operator on~$S$.
 \end{proof}

\section[The families ${\mathbf D}_{\lambda,N}$, $\widetilde {\mathbf D}_{\lambda,N}$ and ${\mathbf D}_N(\lambda)$]{The families $\boldsymbol{{\mathbf D}_{\lambda,N}}$, $\boldsymbol{\widetilde {\mathbf D}_{\lambda,N}}$ and $\boldsymbol{{\mathbf D}_N(\lambda)}$}

For $N\geq 1$, set
\begin{gather*}
\widetilde{\mathbf D}_{\lambda,N} = \widetilde {\mathbf D}_{\lambda+N-1}\circ \cdots \circ \widetilde {\mathbf D}_\lambda.
\end{gather*}
Let $M^N$ be the operator on $C^\infty(S)$ given by multiplication by $x_n^N$. Set
\begin{gather*}
{\mathbb D}_{\lambda, N} = I_{n-N-\lambda} \circ M^N\circ I_\lambda.
\end{gather*}

\begin{Proposition}\quad
\begin{enumerate}\itemsep=0pt
\item[$i)$] $\widetilde {\mathbf D}_{\lambda,N}$ and ${\mathbb D}_{\lambda,N}$ are differential operators on $S$ which intertwine $\pi_{\lambda\vert G'}$ and $\pi_{\lambda+N\vert G'}$.
\item[$ii)$]
\begin{gather}\label{ddtilde}
\widetilde {\mathbf D}_{\lambda,N} = \pi^{n(N-1)}\Gamma(\lambda+N)\Gamma(n-\lambda-N){\mathbb D}_{\lambda, N}.
\end{gather}
\end{enumerate}
\end{Proposition}

\begin{proof} Repeated uses of \eqref{Mintw} show that, for any $\mu\in \mathbb C$, $M^N$ intertwines ${\pi_\mu}_{\vert G'}$ and ${\pi_{\mu-N}}_{\vert G'}$. Hence ${\mathbb D}_{\lambda,N}$ intertwines ${\pi_\lambda}_{\vert G'} $ and ${\pi_{\lambda+N}}_{\vert G'}$. On the other hand, repeated uses of \eqref{intwD} proves that $\widetilde {\mathbf D}_{\lambda,N}$ also intertwines ${\pi_\lambda}_{\vert G'} $ and ${\pi_{\lambda+N}}_{\vert G'}$.

Next, $\widetilde {\mathbf D}_{\lambda, N}$ as a composition of dif\/ferential operators on $S$ is a dif\/ferential operator. So it remains to prove \eqref{ddtilde}.

Substitute ${\mathbf D}_{\lambda+j} = I_{n-\lambda-j-1}\circ M\circ I_{\lambda+j}$ for $0\leq j\leq N-1$ to get
\begin{gather*}{\mathbf D}_{\lambda+N-1} \circ {\mathbf D}_{\lambda+N-2} \circ \dots \circ {\mathbf D}_\lambda\\
\qquad{} =I_{-\lambda+n-N} \circ\dots \circ I_{-\lambda+n-j-1} \circ M\circ I_{\lambda+j}\circ I_{-\lambda+n-j} \circ M \circ I_{\lambda+j-1}\circ \dots\circ I_\lambda,
\end{gather*}
and use \eqref{inverseJ} repeatedly for $\lambda+j$ to obtain
\begin{gather*}{\mathbf D}_{\lambda+N-1} \circ {\mathbf D}_{\lambda+N-2} \circ \dots \circ {\mathbf D}_\lambda\\
\qquad{} = \pi^{n(N-1)} \left(\prod_{j=1}^{N-1}\Gamma(\lambda+j) \Gamma(n-\lambda-j)\right)^{-1}I_{-\lambda+n-N}\circ M^N\circ I_{\lambda}.
\end{gather*}
Multiply by the appropriate $\Gamma$ factors coming from \eqref{widetildeD} to get the formula.
\end{proof}

The group $G'$ acts conformally on $S'$. The scalar principal series for $G'\simeq {\rm SO}_0(1,n)$ is def\/ined as follows: for $\mu\in \mathbb C$, for $g\in G'$ and $f\in C^\infty(S')$,
\begin{gather}\label{pi'}
\pi'_\mu(g) f(x) = \kappa\big(g^{-1},x\big)^\mu f\big(g^{-1}(x)\big),\qquad x\in S'.
\end{gather}
Let $\res\colon C^\infty(S) \longrightarrow C^\infty(S')$ be the restriction map from $S$ to $S'$, def\/ined for $f\in C^\infty(S)$ by $(\res f)(x) = f(x), x\in S'$. The last remark makes clear that for $\lambda\in \mathbb C$ and for $g\in G'$,
\begin{gather}\label{res}
\res \circ \pi_\lambda(g) = \pi'_\lambda(g) \circ \res.
\end{gather}
 Def\/ine the dif\/ferential operator ${\mathbf D}_N(\lambda)\colon C^\infty(S) \longrightarrow C^\infty(S')$ by
\begin{gather*}
{\mathbf D}_N(\lambda) = \res\circ \widetilde {\mathbf D}_{\lambda,N} .
\end{gather*}
\begin{Theorem}
${\mathbf D}_N(\lambda)$ satisfies
\begin{gather*}
\forall\, g\in G'\qquad {\mathbf D}_N(\lambda) \circ \pi_\lambda(g) = \pi'_{\lambda+N}(g)\circ {\mathbf D}_N(\lambda).
\end{gather*}
\end{Theorem}

The proof follows immediately from the covariance property of $\widetilde {\mathbf D}_{\lambda, N}$ and of the restriction map~\eqref{res}.

\section[The family $E_N(\lambda)$]{The family $\boldsymbol{E_N(\lambda)}$}

The previous constructions of dif\/ferential operators made for $S$ and $S'$ can be made in a similar manner in the non compact picture, i.e., for $\mathbb R^n$ and $\mathbb R^{n-1}$. For $N\in \mathbb N$, let $E_{\lambda,N}$ be def\/ined by
\begin{gather*}E_{\lambda, N} = E_{\lambda+N-1}\circ \dots \circ E_\lambda
\end{gather*}
and \begin{gather*}E_N(\lambda) = \res \circ E_{\lambda, N}, \end{gather*}
where $\res$ is the restriction from $\mathbb R^n$ to $\mathbb R^{n-1}$. Then $E_{\lambda,N}$ is a dif\/ferential operator on $\mathbb R^n$ which is covariant with respect to $(\rho_{\lambda \vert G'} , \rho_{\lambda+N \vert G'})$ and $E_N(\lambda)$ is a dif\/ferential operator from~$\mathbb R^n$ to~$\mathbb R^{n-1}$ which is covariant with respect to $(\rho_{\lambda\vert G'} , \rho'_{\lambda+N})$.\footnote{The representation $\rho'$ is the principal series for~$G'$ realized in the $\mathbb R^{n-1}$, def\/ined in analogy with \eqref{pi'}.}

In this section, for the sake of completeness, we compare $E_N(\lambda)$ with Juhl's operator for the non compact model. For $\xi\in \mathbb R^n$, introduce the notation $\xi=(\xi',\xi_n)$ where $\xi'\in \mathbb R^{n-1}$. Let $\Delta'= \sum\limits_{j=1}^{n-1} \frac{\partial^2}{\partial \xi_j^2}$.

\begin{Proposition} Let $E\colon C^\infty(\mathbb R^n)\longrightarrow C^\infty(\mathbb R^{n-1})$ be a differential operator and assume that~$E$ is covariant with respect to $(\rho_{\lambda \vert G'}, \rho'_{\lambda+N})$ for some $N\in \mathbb N$. Then there exits a family of complex constants $a_j$, $0\leq j\leq [\frac{N}{2}]$ such that
\begin{gather*}E=\res\circ\sum_{j=0}^{[\frac{N}{2}]} a_j \left(\frac{\partial}{\partial \xi_n}\right)^{N-2j} \Delta'^j.
\end{gather*}
\end{Proposition}

\begin{proof} By the def\/inition of a dif\/ferential operator from $\mathbb R^n$ to $\mathbb R^{n-1}$, $E$ can be written as a~locally f\/inite sum
\begin{gather*}\sum_{i,J} a_{i,J}(\xi') \res\circ \left(\frac{\partial}{\partial \xi_n}\right)^i\partial^J,
\end{gather*}
where $J=(j_1,j_2,\dots,j_{n-1})$ is a $(n-1)$-tuple, $\partial^J=\prod\limits_{k=1}^{n-1}\big(\frac{\partial}{\partial \xi_k}\big)^{j_k}$ and $a_{i,J}$
is a smooth function of $\xi' \in \mathbb R^{n-1}$.

The invariance by translations forces the $a_{i,J}$ to be constants (and also the sum to be f\/inite). The invariance by ${\rm SO}(n-1)$ forces the expression to be of the form
\begin{gather*}\sum_{i,j} a_{i,j} \left(\frac{\partial}{\partial \xi_n}\right)^i(\Delta')^j
\end{gather*}
and f\/inally the covariance under the action of the dilations forces $i+2j=N$. The statement follows.
\end{proof}

Notice that the proof uses only the covariance property for the parabolic subgroup of af\/f\/ine conformal dif\/feomorphisms of $\mathbb R^{n-1}$. The full covariance condition implies further conditions on the coef\/f\/icients $a_{i,j}$, explicitly written by A.~Juhl (see \cite{j}, condition (5.1.2) for~$N$ even and~(5.1.22) for~$N$ odd), proving in particular that there exists (up to a constant) a unique covariant dif\/ferential operator. Now let
\begin{gather*}E_N(\lambda)=\sum_{j=0}^{[\frac{N}{2}]} a_j (\lambda,N)\left(\frac{\partial}{\partial \xi_n}\right)^{N-2j} \Delta'^j,
\end{gather*}
where $a_j(\lambda, N)$ are complex numbers.

To f\/ind the ratio between $E_N(\lambda)$ and the corresponding Juhl's operator, it is enough to know some coef\/f\/icient of $E_N(\lambda)$ and to compare it to the corresponding coef\/f\/icient of Juhl's operator. It turns out that the coef\/f\/icient $a_0(\lambda,N)$ is rather easy to compute.

\begin{Lemma}\quad
\begin{enumerate}\itemsep=0pt
\item[$i)$] For $k\in \mathbb N$ and $\mu\in \mathbb C$,
\begin{gather*}
E_\mu \xi_n^k = k(2\mu-n+1+k) \xi_n^{k-1}.
\end{gather*}
\item[$ii)$] For $N\in \mathbb N$ and for $\lambda\in \mathbb C$,
\begin{gather*}
E_{\lambda,N} \big(\xi_n^N\big) = N! (2\lambda-n+N+1)(2\lambda-n+N+2)\cdots (N+2\lambda-n+2N).
\end{gather*}
\item[$iii)$] The constant $a_0(\lambda, N)$ is given by
\begin{gather*}
a_0(\lambda,N) = (2\lambda-n+N+1)(2\lambda-n+N+2)\cdots(2\lambda-n+2N).
\end{gather*}
\end{enumerate}
\end{Lemma}

\begin{proof} Let $f$ be a function on $\mathbb R^n$ which depends only on $\xi_n$. Then $\Delta' f = 0$, and
\begin{gather*}E_\mu f = \left((2\mu-n+2)\frac{\partial}{\partial \xi_n} + \xi_n \frac{\partial^2}{\partial \xi_n^2}\right) f,\end{gather*}
so that $i)$ and $ii)$ are reduced to elementary one variable computations. For $iii)$ observe that
\begin{gather*}E_{\lambda, N}\big( \xi_n^N \big)= a_0(\lambda, N) \left(\frac{\partial}{\partial \xi_n}\right)^N \big(\xi_n^N\big)+0+\dots +0= N! a_0(\lambda, N),\end{gather*}
hence $E_N(\lambda) (\xi_n^N) = N! a_0(\lambda, N)$ and $iii)$ follows.
\end{proof}

The comparison with Juhl's operator is then easy. As his normalization depends on the parity of $N$, one has to examine two cases.
\begin{itemize}\itemsep=0pt
\item In the even case, $E_N(\lambda)$ is obtained by multiplying Juhl's operator by
\begin{gather*}\frac{N!}{\big(\frac{N}{2}\big)!} 2^{\frac{N}{2}-1}\prod_{j=1}^{\frac{N}{2}} (2\lambda-n+N+2j).
\end{gather*}
\item In the odd case, $E_N(\lambda)$ is obtained by multiplying Juhl's operator by
\begin{gather*}\frac{N!}{\big(\frac{N-1}{2}\big)!} 2^{\frac{N+1}{2}}\prod_{j=0}^{\frac{N-1}{2}}(2\lambda-n+N+1+2j).
\end{gather*}
\end{itemize}

\section[The operator ${\mathbf D}_\lambda$ in the ambient space model]{The operator $\boldsymbol{{\mathbf D}_\lambda}$ in the ambient space model}

This last section is devoted to another (simpler) construction of (a multiple of) the opera\-tor~${\mathbf D}_\lambda$, using the \emph{ambient space} realization of the principal series.

Let $\Xi^+$ be the positive light cone,
\begin{gather*}\Xi^+=\big\{ \mathbf x \in \mathbf E,\, Q(\mathbf x) =[ \mathbf x,\mathbf x] = 0, \, t(\mathbf x)>0\big\}.
 \end{gather*}
For $\lambda\in \mathbb C$, let
 \begin{gather*}\mathcal H_\lambda=\big\{ F\in C^\infty(\Xi^+),\, F(t\mathbf x)=t^{-\lambda}F(\mathbf x), \text{ for } t\in \mathbb R^+\big\}.
\end{gather*}
The space $\mathcal H_\lambda$ is in one-to-one correspondence with the space $C^\infty(S)$ through the map $R_\lambda$
\begin{gather*}\mathcal H_\lambda\ni F\longmapsto R_\lambda F\in C^\infty(S), \qquad R_\lambda F(x) = F\big((1,x)\big).\end{gather*}
The space $\mathcal H_\lambda$ inherits the corresponding topology. For $g\in G$, and $F\in \mathcal H_\lambda$, let
\begin{gather*}\Pi_\lambda(g)F=F\circ g^{-1}.
\end{gather*}
Then $\Pi_\lambda$ def\/ines a representation of $G$ on $\mathcal H_\lambda$ and it is easily verif\/ied that
\begin{gather}\label{intwR}
R_\lambda\circ\Pi_\lambda(g)= \pi_\lambda(g)\circ R_\lambda,
\end{gather}
so that $\Pi_\lambda$ is yet another model for the representation $\pi_\lambda$ of $G$.

Let $\square =\frac{\partial^2}{\partial t^2}-\sum\limits_{j=0}^n \frac{\partial^2}{\partial x_j^2}$ be the d'Alembertian on~$\mathbf E$. It satisf\/ies, for any $g\in G$ and~$F$ a~smooth function on~$\mathbf E$
\begin{gather}\label{invsquare}
\square(F\circ g) = (\square F)\circ g.
\end{gather}
The following lemma, which I learnt from \cite{es} is a key result for what follows.

\begin{Lemma}\looseness=1 Let $F_1$, $F_2$ be two smooth functions defined in a neighborhood of~$\Xi^+$, positively homogeneous of degree $-\frac{n}{2}+1$ and which coincide on~$\Xi^+$. Then $\square F_1$ and $\square F_2$ coincide on~$\Xi^+$.
\end{Lemma}
\begin{proof} The function $F_1-F_2$ vanishes on $\Xi^+$. Notice that $dQ(\mathbf x)\neq 0$ for any $\mathbf x\in \Xi^+$. Hence, there exists a smooth function $G$ def\/ined on a neighborhood of $\Xi^+$ such that \begin{gather*}F_1(\mathbf x)-F_2(\mathbf x) = Q(\mathbf x) G(\mathbf x).\end{gather*}
Moreover, \looseness=1 $G$ is positively homogeneous of degree~$-\frac{n}{2}-1$. Now, for any smooth function~$H$ on~$\mathbf E$
\begin{gather*}\square(QH) = 2(n+2)H +4EH +Q \square H,
\end{gather*}
where $E= t\frac{\partial}{\partial t}+\sum\limits_{j=0}^n x_j\frac{\partial}{\partial x_j}$ is the Euler operator. As~$G$ is homogeneous of degree \smash{$-\frac{n}{2}-1$},
\begin{gather*}EG(\mathbf x) = \left(-\frac{n}{2}-1\right)G(\mathbf x),
\end{gather*}
and hence $\square (QG) (x) = 0$ for $x\in \Xi^+$. The lemma follows.
\end{proof}

The next result is a reformulation of the previous lemma.
\begin{Lemma}\label{Yam}
Let $F\in \mathcal H_{\frac{n}{2}-1}$. Extend $F$ smoothly to a positively homogeneous function of degree $-\frac{n}{2}+1$ to neighborhood of~$\Xi^+$. Then the restriction to $\Xi^+$ of $\square F$ does not depend on the extension.
\end{Lemma}
 The operator $\square$ induces a map from $\mathcal H_{\frac{n}{2}-1}$ to $\mathcal H_{\frac{n}{2}+1}$ and intertwines the action of~$G$. Let $\Delta_S$ be the operator def\/ined on $C^\infty(S)$ by
\begin{gather*}\Delta_S = R_{\frac{n}{2}+1}\circ \square\circ R^{-1}_{\frac{n}{2}-1}.
\end{gather*}
The invariance of $\square$ (see \eqref{invsquare}) and the covariance of $R_\lambda$ (see~\eqref{intwR}) imply the following proposition.
\begin{Proposition} The operator $\Delta_S$ $($conformal Laplacian or Yamabe operator on~$S)$ is a~differential operator on~$S$ which is covariant with respect to $(\pi_{\frac{n}{2}-1}, \pi_{\frac{n}{2}+1})$.
\end{Proposition}

Let $\mathbf B_\mu$ be the dif\/ferential operator on $\mathbf E$ def\/ined by
\begin{gather*}
\mathbf B_\mu F(\mathbf x) = x_n\square F(\mathbf x) -2\mu \frac{\partial F}{\partial x_n}.
\end{gather*}

\begin{Lemma} Let $\mu\in \mathbb C$. Let $F$ be a smooth function on $\mathbf E$. Then, on $\{ x_n\neq 0\}$,
\begin{gather} \label{cmu2}
\mathbf B_\mu F(\mathbf x) = x_n \vert x_n\vert^{-\mu} \square \big(\vert x_n\vert^\mu F\big)(\mathbf x) +\mu(\mu-1) \frac{1}{x_n} F(\mathbf x).
\end{gather}
\end{Lemma}

\begin{proof} By an elementary calculation,
\begin{gather*}\square \big(\vert x_n\vert^\mu F\big)(\mathbf x) =\vert x_n\vert^\mu \square F(\mathbf x)-2\mu \sgn(x_n) \vert x_n\vert^{\mu-1} \frac{\partial F}{\partial x_n}(\mathbf x)-\mu(\mu-1) \vert x_n\vert^{\mu-2} F(\mathbf x),
\end{gather*}
so that
\begin{gather*}
\square\big(\vert x_n\vert^\mu F\big) +\mu(\mu-1) \vert x_n\vert^{\mu-2} F = \sgn(x_n) \vert x_n\vert^{\mu-1} \mathbf B_\mu F.
\end{gather*}

The conclusion follows, by noticing that $x_n=\sgn(x_n)\vert x_n\vert$.
\end{proof}

\begin{Proposition}\label{invCmu}
 Let $g\in G'$. Then for $F$ a smooth function on $\mathbf E$,
\begin{gather*}\mathbf B_\mu(F\circ g) = (\mathbf B_\mu F)\circ g.
\end{gather*}
\end{Proposition}
\begin{proof} As $g\in G'$, the coordinate $x_n$ is unchanged by the action of $g$, and the action of~$g$ commutes with $\frac{\partial}{\partial x_n}$ and with~$\square$. The result follows.
\end{proof}

\begin{Proposition} Let $F\in \mathcal H_\lambda$. Extend $F$ smoothly to a neighborhood of $\Xi^+$ as a positively homogeneous function of degree $-\lambda$. Then the restriction to $\Xi^+$ of $\mathbf B_{\lambda-\frac{n}{2}+1}F$
does not depend on the extension.
\end{Proposition}
\begin{proof}
The function $\vert x_n\vert^{\lambda-\frac{n}{2}+1} F(\mathbf x)$ is homogenous of degree $-\frac{n}{2}+1$, and hence, by Lemma~\ref{Yam}, for $x\in \Xi^+$, $\square\big( \vert x_n\vert^{\lambda-\frac{n}{2}+1} F\big)(\mathbf x)$ only depend on the values of $F$ on $\Xi^+$. Hence, by~\eqref{cmu2}, for~$\mathbf x$ in~$\Xi^+$, $x_n\neq 0$, $\mathbf B_{\lambda-\frac{n}{2}+1}F(\mathbf x)$ does not depend on the extension of~$F$. The result follows by continuity.
\end{proof}

\begin{Proposition}The differential operator $\mathbf B_{\lambda-\frac{n}{2}+1}$ induces a map from $\mathcal H_\lambda$ into $\mathcal H_{\lambda+1}$, which commutes with the action of~$G'$.
\end{Proposition}
\begin{proof}
The invariance follows from Proposition~\ref{invCmu}.
\end{proof}

Having constructed a covariant operator in the ambient space model, it is possible to express it both in the non-compact and in the compact picture.

\begin{Proposition}\label{Dlambdanc}
 The local expression of the operator $\mathbf B_{\lambda-\frac{n}{2}+1}$ in the non compact picture is equal to $-E_\lambda$.
\end{Proposition}
\begin{proof}
Let $f$ be a smooth function on $\mathbb R^n$. Recall the map $c$ (cf.~\eqref{explicitc}) which realizes the passage from $\mathbb R^n$ to $S$. Its inverse is given by
\begin{gather*}S\setminus\{-\mathbf 1\} \ni (x_0,x_1,\dots, x_n)\longmapsto \left(\frac{x_1}{1+x_0},\dots, \frac{x_n}{1+x_0}\right).
\end{gather*}

So map $f$ to a function on $S$ by
\begin{gather*}C_\lambda^{-1} f(x) = (1+x_0)^{-\lambda}f\left(\frac{x_1}{1+x_0}, \dots,\frac{x_n}{1+x_0}\right) .
\end{gather*}
Consider the function $F$ on $\mathbf E$ def\/ined by
\begin{gather*}F(\mathbf x) = (t+x_0)^{-\lambda} f\left(\frac{x_1}{t+x_0}, \dots,\frac{x_n}{t+x_0}\right) .
\end{gather*}
Then $F$ is homogenous of degree $-\lambda$ and coincide on $S$ with $C_\lambda^{-1} f$. To compute $\mathbf B_{\lambda -\frac{n}{2}+1}F$, f\/irst observe that
\begin{gather*}\frac{\partial F}{\partial t} = \frac{\partial F}{\partial x_0},\qquad \frac{\partial^2 F}{\partial t^2} = \frac{\partial^2F}{\partial x_0^2},
\end{gather*}
so that
\begin{gather*}\square F = -\sum_{j=1}^n \frac{\partial^2F}{\partial x_j^2}. \end{gather*}
Hence
\begin{gather*}\mathbf B_{\lambda -\frac{n}{2}+1}F(\mathbf x) = -(t+x_0)^{-\lambda-2}x_n (\Delta f) \left(\frac{x_1}{t+x_0}, \dots,\frac{x_n}{t+x_0}\right)\\
\hphantom{\mathbf B_{\lambda -\frac{n}{2}+1}F(\mathbf x) =}{}
-2\left(\lambda-\frac{n}{2}+1\right)(t+x_0)^{-\lambda} (t+x_0)^{-1}\frac{\partial f}{\partial \xi_n}\left(\frac{x_1}{t+x_0}, \dots,\frac{x_n}{t+x_0}\right).
\end{gather*}
Now letting $\mathbf x = (1,c(\xi))$,
\begin{gather*}\mathbf B_{\lambda -\frac{n}{2}+1}F(1,c(\xi)) =
- \xi_n \Delta f (\xi) -(2\lambda-n+2) \frac{\partial f}{\partial \xi_n} (\xi).
\end{gather*}
A comparison with \eqref{Dlambda} implies the result.
\end{proof}

\begin{Proposition}\label{Dlambdac}
 The expression of the operator $\mathbf B_{\lambda-\frac{n}{2}+1}$ on $S$ is given by
\begin{gather*}
 x_n\vert x_n\vert^{-\lambda+\frac{n}{2}-1} \Delta_S\circ \vert x_n\vert^{\lambda-\frac{n}{2}+1} + \left(\lambda-\frac{n}{2} +1\right)\left(\lambda-\frac{n}{2}\right)\frac{1}{x_n} .
\end{gather*}
The expression, a priori defined on $x_n\neq 0$ can be continued continuously to all of~$S$.
\end{Proposition}

\begin{proof}
Let $f\in C^\infty(S)$. Then
\begin{gather*}F(\mathbf x) = \big(x_0^2+\dots + x_n^2\big)^{-\lambda} f\left(\frac{x_0}{\sqrt{x_0^2+\dots +x_n^2}},\dots, \frac{x_n}{\sqrt{x_0^2+\dots +x_n^2}} \right)
\end{gather*}
is a function def\/ined on $\mathbf E\setminus \{0\}$ which is positively homogeneous of degree $-\lambda$ and such that for $x\in S$,
\begin{gather*}F(1,x) = f(x).
\end{gather*}
By \eqref{cmu2} with $\mu=\lambda-\frac{n}{2}+1$ and for $\mathbf x\neq 0$, $x_n\neq 0$,
\begin{gather*}\mathbf B_{\lambda-\frac{n}{2}+1} F (\mathbf x) =x_n \vert x_n\vert^{-\lambda+\frac{n}{2} +1} \square \big(\vert x_n\vert^{\lambda-\frac{n}{2}+1} F\big)(\mathbf x)+\left(\lambda-\frac{n}{2} +1\right)\left(\lambda-\frac{n}{2}\right)\frac{1}{x_n} F(\mathbf x).
\end{gather*}
The function $\vert x_n\vert^{\lambda-\frac{n}{2}+1}F(\mathbf x)$ is positively homogeneous of degree $-\frac{n}{2}+1$. Thus, by Lemma~\ref{Yam} and the def\/inition of the Yamabe operator $\Delta_S$, for $x\in S$,
\begin{gather*}\mathbf B_{\lambda-\frac{n}{2}+1} F (1,x)= x_n\vert x_n\vert^{-\lambda+\frac{n}{2}-1} \Delta_S\big(\vert x_n\vert^{\lambda-\frac{n}{2}+1} f\big) (x) + \left(\lambda-\frac{n}{2} +1\right)\left(\lambda-\frac{n}{2}\right)\frac{1}{x_n} f(x),
\end{gather*}
from which the statement follows, at least for $x_n\neq 0$. As $\mathbf B_{\lambda-\frac{n}{2}+1}$ induces a smooth dif\/ferential operator on~$S$, the formula determines the operator on all of~$S$ by continuity.
\end{proof}

\subsection*{Acknowledgements}

It is a pleasure to thank the anonymous referees for their contributions which helped to improve and reshape the initial version of this article.

\pdfbookmark[1]{References}{ref}
\LastPageEnding

\end{document}